\documentclass{amsproc}

\usepackage{amssymb,amscd,amsmath,latexsym}
\usepackage[mathcal]{euscript}
\usepackage{hyperref}

\newtheorem{thm}{Theorem}[section]
\newtheorem{cor}[thm]{Corollary}
\newtheorem{lem}[thm]{Lemma}
\newtheorem{prop}[thm]{Proposition}
\theoremstyle{definition}

\newtheorem{exm}[thm]{Example}
\newtheorem{conj}[thm]{Conjecture}
\newtheorem{prob}[thm]{Problem}

\theoremstyle{remark}

\newcommand{\pbp}{$p  \times p$ }

\newcommand{\Ext}{\operatorname{Ext}}

\newcommand{\Hom}{\operatorname{Hom}}

\newcommand{\sgn}{\operatorname{sgn}}

\newcommand{\HH}{\operatorname{H}}

 \numberwithin{equation}{section}

\begin{document}

\title[Frobenius twists]{``Frobenius twists"  in the representation theory of the symmetric group.}

\author{\sc David J. Hemmer}
\address{{Department of Mathematics\\ University at Buffalo, SUNY \\
244 Mathematics Building\\Buffalo, NY~14260, USA}}

\email{dhemmer@math.buffalo.edu}
\thanks{Research of the  author was supported in part by NSA
grant  H98230-10-0192}

\subjclass[2010]{Primary 20C30}

\date{April 2, 2012}

\begin{abstract}For the general linear group $GL_n(k)$ over an algebraically closed field $k$ of characteristic $p$, there are two types of ``twisting" operations that arise naturally on partitions. These are of the form $\lambda \rightarrow p\lambda$ and $\lambda \rightarrow \lambda + p^r\tau$ The first comes from the Frobenius twist, and the second arises in various tensor product situations, often from tensoring with the Steinberg module. This paper surveys and adds to an intriguing series of seemingly unrelated symmetric group results where this partition combinatorics arises, but with no structural explanation for it. This includes cohomology of simple, Specht and Young modules, support varieties for Specht modules, homomorphisms between Specht modules, the Mullineux map, $p$-Kostka numbers and tensor products of Young modules.
\end{abstract}

\maketitle

\section{Introduction}
\label{sec: Introduction}
 Let $k$ be an algebraically closed field of characteristic $p$. An important construction in the representation theory of the general linear group $G:=GL_n(k)$ is the Frobenius twist, which takes a $G$ module $M$ to the module $M^{(1)}$.  The action of $G$ on $M^{(1)}$ is as on $M$ except twisted by the Frobenius endomorphism $F: G \rightarrow G$, which raises each matrix entry to the $p$th power. Probably the most important $G$ modules are the Steinberg modules $\operatorname{St}_r=L((p^r-1)\rho)$. For example the operation of ``twist then tensor with $\operatorname{St}_r$" plays a key role in the proof of Kempf's vanishing theorem.

 In the last decade or so there have been a great variety of results and conjectures on the symmetric group $\Sigma_d$ that ``look like" they should come from doing a Frobenius twist or taking a tensor product with a Steinberg module, even though neither construction has any reasonable analogue in the world of $k\Sigma_d$ modules. Regular twisting results involve partitions $\lambda, p\lambda, p^2\lambda$, etc... Results reminiscent of twisting then tensor with $\operatorname{St}_r$ could relate $\lambda$ with $\lambda + p^r\tau.$ Both results we informally think of as twisting type theorems, keeping in mind again that there is no Frobenius twist for $k\Sigma_d$-modules.

 In this paper we survey the known results of this type, and add a couple more new results together with new examples, conjectures and a variety of open problems that remain. Particularly striking is the array of techniques that arise in the proofs of the various results. Twisting behavior seems to arise in many different ways. It is probably naive to expect some kind of uniform ``Frobenius twist" for symmetric groups that captures these diverse results. We will assume background information for $k\Sigma_d$ representation theory as found in \cite{JamesKerberbook}, and use the same notation.

Let $\lambda=(\lambda_1, \lambda_2, \ldots, \lambda_n)$ be a partition of $d$ with at most $n$ parts. These partitions correspond to dominant polynomial weights for $G$ and many natural $G$ modules are labeled by them; for example, irreducible modules $L(\lambda)$, Weyl modules $V(\lambda)$ and induced modules $\HH^0(\lambda)$. The interested reader will find Jantzen's book \cite{jantzenbook2nded} a definitive although likely unnecessary reference, as this paper will focus on the symmetric group.

For a $G$ module $M$, let $M^{(r)}$ denote the $r$th Frobenius twist \cite[I.9.10]{jantzenbook2nded} of $M$.  A special case of the Steinberg Tensor Product Theorem \cite[II.3.17]{jantzenbook2nded} gives that:

\begin{equation}
\label{eq: twistofsimplemodule}
L(\lambda)^{(1)} \cong L(p\lambda),
\end{equation}

\noindent where $p\lambda:=(p\lambda_1, \ldots, p\lambda_n) \vdash pd$. Equation \ref{eq: twistofsimplemodule} suggests the operation $\lambda \rightarrow p\lambda$ is quite natural for $G$, and it is not surprising to encounter theorems involving these ``twisted" partitions. For example the isomorphism
\begin{equation}
\label{eq: H1stabilizes Llambda} \HH^1(G, L(\lambda)) \cong \HH^1(G, L(p\lambda))
\end{equation}
\noindent is a  special case of \cite[Thm. 7.1]{CPSvdKrationalandgeneric} and is realized explicitly on the level of short exact sequences by applying the Frobenius twist.

Let $\rho=\rho_n$ denote the partition $(n-1,n-2, \ldots, 2, 1,0)$. Another operation that arises frequently in the representation theory of $G$ takes a partition $\lambda$ to $\lambda+p^r\tau$ for some other partition $\tau$, where $\lambda$ often involves the so-called Steinberg weight $(p^r-1)\rho$. For example \cite[II.3.19]{jantzenbook2nded}:
\begin{equation}
\label{eq: steinbergforinduced}
\HH^i((p^r-1)\rho) \otimes \HH^i(\tau)^{(r)} \cong \HH^i((p^r-1)\rho + p^r\tau).
\end{equation}

These results are also quite natural as $\HH^i((p^r-1)\rho)$ is the ubiquitous Steinberg module $\operatorname{St}_r$, which is simple and both projective and injective as a module for the Frobenius kernel $G_r$.

Turning our attention to $k\Sigma_d$, we again find modules labeled by partitions, this time by all partitions of $d$, not just those with at most $n$ parts. For example we have Specht modules $S^\lambda$, Young modules $Y^\lambda$, irreducible modules $D^\lambda$ for $\lambda $ $p$-regular, etc. However there is no analogue of the Frobenius twist. Moreover $p\lambda$ is a partition of $pd$, and so, for example,  $S^\lambda$ and $S^{p\lambda}$ are modules for different groups with no apparent connection. Nevertheless over the last ten years or so there have been numerous symmetric group results involving this kind of ``twisting" of partitions, and the proofs use an impressive variety of different techniques. Other results are reminiscent of \eqref{eq: steinbergforinduced}, even though there is no natural analogue of the Steinberg module for $\Sigma_d$.

\section{Schur subalgebras and the original ``twist".}
\label{sec: Morita}
We believe the first appearance of ``twisting" type results for the symmetric group arose in the thesis of Henke, published in part in the paper \cite{HenkeSchursubalgebras}. (Although James' computation of decomposition numbers for two-part partitions \cite[24.15]{Jamesbook} can be put in similar form). For example Henke proved:

\begin{thm}
\label{thm: Henke2part}
Fix $d$ and let $\lambda=(d-k, k)$ be a $p$-regular partition.  Then there is an $a \geq 1$ such that there exists a strong submodule lattice isomorphism between $S^{(d-k, k)}$ and $S^{(d-k+cp^a, k)}$ for any $c \geq 1$ such that $cp^a$ is even. Similar lattice isomorphisms exist for Young modules and permutation modules.
\end{thm}

More general results along the same line were obtain in \cite{HenkeKoenigRelatingPolynomial}, involving adding a large power of $p$ to the first part of a partition and obtaining equality of decomposition numbers and $p$-Kostka numbers. We should warn though that Section 4 of \cite{HenkeKoenigRelatingPolynomial} has a gap.
\footnote{For example $p=2, \lambda=(4,3,1), \mu=(8), p^d=4$ gives a counterexample to Cor. 4.5} The problem is that the set of weights considered, for example in Cor. 4.1 and following, is not an ideal or coideal.

The proofs of these results involve constructing explicit isomorphisms between generalized Schur algebras in different degrees. The numerical equalities obtained are then translated to the symmetric group setting.

\section{Generic cohomology}
Only more recently have symmetric group results relating $\lambda$ and $p\lambda$ appeared. Many of these results are motivated by or suggestive of the famous generic cohomology theorem from \cite{CPSvdKrationalandgeneric}, which we describe briefly now.
For $G$ modules $M_1$ and $M_2$, there is a natural map
\begin{equation}
    \label{eq: injectiononextfrobtwist}
\Ext^i_G(M_1,M_2) \rightarrow \Ext^i_G(M_1^{(r)}, M_2^{(r)})
\end{equation}
induced by applying the Frobenius twist to the corresponding exact sequences of $G$ modules. The map \eqref{eq: injectiononextfrobtwist} is always an injection \cite[II.10.14]{jantzenbook2nded}. Thus, for fixed $i$, we have a sequence of injective maps
\begin{equation}
\label{eq: genericcohomologysequence}
\HH^i(G,M) \rightarrow \HH^i(G,M^{(1)}) \rightarrow \HH^i(G, M^{(2)}) \rightarrow \cdots.
\end{equation}

When $M$ is finite-dimensional, the sequence \eqref{eq: genericcohomologysequence} is known \cite{CPSvdKrationalandgeneric} to stabilize, and the limit is called the \emph{generic cohomology} $\HH^i_{gen}(G,M)$ of $M$. Since $G=GL_n(k)$, the sequence \eqref{eq: genericcohomologysequence} is known to stabilize immediately for $i=1$, \cite[Thm. 7.1]{CPSvdKrationalandgeneric} i.e.
\begin{equation}
\label{eq: generic1cohostabilizesimmediately}
\HH^1_{\operatorname{gen}}(G,M) \cong \HH^1(G,M).
\end{equation}
Note that \eqref{eq: H1stabilizes Llambda} is a special case. We call theorems relating cohomology of symmetric group modules  $U^\lambda$ and $U^{p\lambda}$ ``generic cohomology" or ``stability" type theorems, where $U$ can be $S, M, D, Y$. For example, a cohomology result relating $S^{p\lambda}$ and $S^{p^2\lambda}$ would be called a generic cohomology type theorem, although we should warn that \eqref{eq: twistofsimplemodule} does not hold for other natural $G$ modules. For example $\HH^0(\lambda)^{(1)}$ is always a proper submodule of $\HH^0(p\lambda).$

\section{Young modules}

For a partition $\lambda \vdash d$ there is a corresponding Young subgroup $\Sigma_\lambda \leq \Sigma_d$ and the permutation module on the cosets is denoted $M^\lambda$. The isomorphism classes of indecomposable summands of these permutation modules are also indexed by partitions of $d$ and are called Young modules, denoted $Y^\lambda$. These are important and well-studied modules. For example the set $\{Y^\lambda \mid \lambda \text{ is $p$-restricted}\}$ is a complete set of projective indecomposable $k\Sigma_d$ modules. Each $Y^\lambda$ is self-dual.  Section 4.6 of \cite{Martinbook} is a good basic reference for Young modules.

\subsection{$p$-Kostka numbers and tensor products}

We have:

\begin{equation}
\label{eq: Mlambdadecomposes}
M^\lambda \cong Y^\lambda \oplus \bigoplus_{\mu > \lambda}  [M^\lambda : Y^\mu] Y^\mu
\end{equation}
where $>$ is the lexicographic order on partitions.
The multiplicities $[M^\lambda : Y^\mu]$ in \eqref{eq: Mlambdadecomposes} are known as $p$-Kostka numbers. As we will see there are results for Young modules of both types, relating $Y^\lambda$ with $Y^{p\lambda}$ and with $Y^{\lambda + p^r \tau}.$

In fact just considering $p$-Kostka numbers, both types of results arise. In \cite{HenkeOnpKostka}, Henke determines completely the $p$-Kostka numbers when $\lambda$ has two parts. She also obtains:

\begin{thm} \cite[Theorem 6.1]{HenkeOnpKostka}
\label{thm: HenkepKostka}
Let $\lambda \vdash d$ and suppose $\lambda_1 \geq d/2$ and $\lambda_2 < p^r$. Then:
$$[M^{\lambda+(ap^r)} : Y^{\mu + (ap^r )}] = [M^{\lambda} : Y^{\mu}]$$ for every $a \geq 1$ and $\mu \vdash d$.
\end{thm}

 Her proof uses the well-known multiplicity formula of Klyachko which gives a kind of recursion for $p$-Kostka numbers in terms of those for smaller partitions, where the assumptions in Theorem \ref{thm: HenkepKostka} ensures those decompositions  are closely related.

More recently in his 2011 thesis Gill proved a strengthened result:

\begin{thm}\cite[Theorem 2.24]{GillThesis} Let $\lambda, \mu \vdash d$ and $a \geq 1$. Suppose $\mu$ has $p$-adic expansion $\mu=\sum_{i=0}^s\mu(i) p^i.$ If $p^r > \operatorname{max} (p^s, \lambda_2)$, then:

$$[M^{\lambda+(ap^r)} : Y^{\mu + (ap^r )}] = [M^{\lambda} : Y^{\mu}].$$
\end{thm}

Gill's techniques include an extensive analysis of Young vertices and the Brou\'{e} correspondence for $p$-permutation modules. This approach to studying Young modules was pioneered by Erdmann in \cite{ErdmannYoungmodulesforsymmetric}.  The twisting $\lambda \rightarrow p\lambda$ behaves very well with respect to the Young vertices, which Gill used to prove the following  stability result on $p$-Kostka numbers under twisting:

\begin{thm}\cite[Theorem 2.21]{GillThesis}
\label{thm: GillpKostkastability}
Let $\lambda, \mu \vdash d$. Then:
$$[M^\lambda : Y^\mu]= [M^{p\lambda} : Y^{p\mu}].$$
\end{thm}

We remark that an alternative proof of Theorem \ref{thm: GillpKostkastability} could be given using the general linear group and an actual Frobenius twist. This is because $p$-Kostka numbers are  equal to weight space multiplicities in simple $GL_n(k)$ modules. Namely:

\begin{prop}\cite[p. 55]{Donkingtiltingalgebraic}
\label{prop: weight space}
The $p$-Kostka number $[M^\lambda : Y^\mu]$ is the dimension of $L(\mu)^\lambda$, the $\lambda$ weight space in the simple module $L(\mu)$.
\end{prop}

But there is no obvious way to use transfer the proof of Proposition \ref{prop: weight space} to give a short symmetric group proof of Theorem \ref{thm: GillpKostkastability}.

\subsection{Young module cohomology} If one thinks of $Y^{p\lambda}$ as a twist of $Y^\lambda$ then the following theorem can be interpreted as a generic cohomology theorem for Young modules, valid in arbitrary degree:

\begin{thm}\cite[Theorem 13.2.1]{CHNYoungmodulecohomology}
\label{thm: Youngmodulegenericcohomology} Fix $i>0$ and let $p$ be
arbitrary. Then there exists $s(i)>0$ such that for any $d$ and
$\lambda \vdash d$ we have
$$\HH^i(\Sigma_{p^{a}d},Y^{p^{a}\lambda}) \cong \HH^i(\Sigma_{p^{a+1}d},Y^{p^{a+1}\lambda}) $$ whenever $a \geq s(i)$.
\end{thm}

Here we have our first example of what will be a common theme, a vector space isomorphism of cohomology but without any map realizing the isomorphism. Indeed the proof of Theorem \ref{thm: Youngmodulegenericcohomology} proceeds by using Schur functor techniques to translate the symmetric group cohomology problem to a representation theory problem for $G$. That in turn is solved using powerful algebraic topology techniques, and ``reverse engineering" the answer gives the isomorphism; but one cannot trace back to find a symmetric group proof or explicit realization of the isomorphism. For example if $p=2$ and $i=1$ we can compute $s(i)=1$, i.e. \cite[Theorem 12.4.1]{CHNYoungmodulecohomology} gives:

\begin{equation}
\label{eq: youngp=2iso}
\Ext^1_{\Sigma_{2d}}(k, Y^{2\lambda}) \cong \Ext^1_{\Sigma_{4d}}(k, Y^{4\lambda}).
\end{equation}
This leads us to ask:
\begin{prob}
\label{prob: realizingisoYoungcoho} Can you prove \eqref{eq: youngp=2iso} purely using symmetric group representation theory? Can you give an explicit map that takes a short exact sequence $0 \rightarrow Y^{2\lambda} \rightarrow U \rightarrow k \rightarrow 0$ and produces the corresponding one for $Y^{4\lambda}$?
\end{prob}

Recall that $\HH^i(G, M) \cong \Ext^i_G(k, M)$. This suggests a natural generalization of Theorem \ref{thm: Youngmodulegenericcohomology} would be to consider $\Ext^i_{\Sigma_d}(Y^\lambda, Y^\mu) \cong \HH^i(\Sigma_d, Y^\mu \otimes Y^\lambda)$.

Already the $i=0$ case of this problem is extremely difficult. Indeed knowing the dimension of $\operatorname{Hom}_{k\Sigma_d}(Y^\lambda, Y^\mu)$ for
all $\lambda, \mu \vdash d$ is equivalent \cite[Proposition 9.2.1]{CHNYoungmodulecohomology} to knowing the decomposition matrix for the Schur algebra $S(d,d)$, and thus contains more information than computing the decomposition matrix for the symmetric group, a notoriously intractable problem. While computing the actual dimensions is beyond reach, Gill managed to prove:

\begin{thm}\cite[Theorem 4.3]{GillTensorProduct}
\label{thm: Gill homYoungmodule}
Let $\lambda, \mu \vdash d$. Then:

$$\dim_k\Hom_{k\Sigma_{d}}(Y^{\lambda}, Y^{\mu}) \leq
\dim_k\Hom_{k\Sigma_{pd}}(Y^{p\lambda}, Y^{p\mu}).$$
\end{thm}
Mackey's theorem easily implies that $Y^\lambda \otimes Y^\mu$ is a direct sum of Young modules. The proof of Theorem \ref{thm: Gill homYoungmodule} uses fact that $\dim_k\Hom_{k\Sigma_{d}}(Y^{\lambda}, Y^{\mu})$ is the number of summands in $Y^\lambda \otimes Y^\mu$ which have a trivial submodule. Since $Y^\lambda \subseteq M^\lambda$ it is clear that $\Hom_{k\Sigma_d}(k, Y^\lambda)$ is at most one-dimensional. The $\lambda$ for which it is nonzero are known (see \cite[Proposition 12.1.1]{CHNYoungmodulecohomology} for example), and these partitions are preserved under twisting. Then the following key theorem from \cite{GillTensorProduct} is used to complete the proof.

\begin{thm}\cite[Thoerem 3.6]{GillTensorProduct}
\label{thm: Gill tensorproduct}
Let $\lambda, \mu, \tau \vdash d$. Then:
$$[Y^\lambda \otimes Y^\mu : Y^\tau] = [Y^{p\lambda} \otimes Y^{p\mu} : Y^{p\tau}].$$
\end{thm}

This theorem is proved numerically by counting multiplicities, but looks like it should come from some explicit twist map!

Several obvious questions arise:

\begin{prob}
\label{prob: is Gill an isomorphism}
Theorem \ref{thm: Gill homYoungmodule} implies a sequence of injections
$$0 \rightarrow \Hom_{k\Sigma_d}(Y^\lambda, Y^\mu) \rightarrow \Hom_{k\Sigma_{pd}}(Y^{p\lambda}, Y^{p\mu}) \rightarrow \Hom_{k\Sigma_{p^2d}}(Y^{p^2\lambda}, Y^{p^2\mu}) \rightarrow \cdots.$$
Does this sequence stabilize? How many twists does it take to do so?
\end{prob}

\begin{prob}
\label{prob: realizeGillmap}
Can one find an explicit embedding of $\Hom_{k\Sigma_d}(Y^\lambda, Y^\mu)$  into $ \Hom_{k\Sigma_{pd}}(Y^{p\lambda}, Y^{p\mu})$?
\end{prob}

\begin{prob}
\label{prob: Gill for higher ext} Can Theorem \ref{thm: Gill homYoungmodule} be extended to $\Ext^i_{k\Sigma_d}(Y^\lambda, Y^\mu)$ for $i>0$, perhaps with more twists required as $i$ grows in the spirit of Theorem \ref{thm: Youngmodulegenericcohomology}?
\end{prob}

\section{Specht modules}
\label{sec: Specht modules}
The Specht modules $S^\lambda$ are perhaps the most well-studied among all $k\Sigma_d$ modules.  They are a complete set of irreducible ${\mathbb C}\Sigma_d$ modules, but they are defined over any field and are not well understood over $k$. For example only quite recently was it proven which remain irreducible over $k$ \cite{FayersirredSpechttypeA}. Computing the homomorphism space $\Hom_{k\Sigma_d}(S^\lambda, S^\mu)$ is an active area of research. Cohomology $\HH^i(\Sigma_d, S^\lambda)$ was worked out in degree $i=0$ more than thirty years ago in \cite[24.4]{Jamesbook}, but the $i=1$ case remains open. Recent results for Specht modules involve both twisting $\lambda \rightarrow p\lambda$ and $\lambda \rightarrow \lambda+p^a\tau$.

\subsection{Homomorphisms between Specht modules and decomposable Specht modules}There is quite a large literature on homomorphisms between Specht modules, for example Carter-Payne maps, row removal theorems, etc.  When $p>2$ it is known that $S^\lambda$ is indecomposable and $\Hom_{k\Sigma_d}(S^\lambda, S^\lambda) \cong k$. In 1980 Murphy \cite{MurphyOnDecomposability} analyzed the hook Specht modules $S^{(d-r, 1^r)}$ in characteristic $p=2$ and discovered they can have arbitrarily many indecomposable summands, so the dimension of $\Hom_{k\Sigma_d}(S^\lambda, S^\lambda)$ can be arbitrarily large. Only in 2011 were such homomorphism spaces with dimension larger than one discovered in odd characteristic \cite{DodgeLargeDimensionHomoSpace}, \cite{LyleLargeDimension}. Dodge's examples are found in Rouquier blocks while Lyle finds explicitly examples with dimension two, then uses row and column removal theorems to get arbitrary dimension. In line with the theme of this paper we observe that the examples from \cite[Theorem 1.2]{LyleLargeDimension} are of the form:

\begin{equation}
\label{eq: Lyleiso}
\Hom_{k\Sigma_{d+3ap}}(S^{\lambda + a(p,p,p)}, S^{\mu + a(p,p,p)}).
\end{equation}

Lyle suspects the spaces in \ref{eq: Lyleiso} are all two-dimensional but does not prove this. If so it would give another example of the $\lambda \rightarrow \lambda + p^r\tau$ twisting. However she constructs the maps individually for each choice of $a$ rather than, for example, by ``twisting" the $a=1$ case, so this is somewhat speculative at this point.

We collect Murphy's results below. Recall from \cite[6.7]{Jamesbook} that $S^\lambda \otimes \sgn \cong (S^{\lambda '})^*$. Since the sign representation is trivial in characteristic two, the assumption $d \geq 2r$ below does not impose a real restriction, all possible hooks are handled.

\begin{thm}
\label{thm: Murphyomnibus}Let $p=2$ and assume $d \geq 2r$.
\begin{enumerate}
  \item If $d$ is even then $S^{(d-r, 1^r)}$ is indecomposable and $\dim \Hom_{k\Sigma_d}(S^{(d-r, 1^r)}, S^{(d-r, 1^r)})=1.$
  \item If $d$ is odd and $r$ is even then $\dim \Hom_{k\Sigma_d}(S^{(d-r, 1^r)}, S^{(d-r, 1^r)})=r/2.$
  \item  If $d$ is odd and $r$ is odd then $\dim \Hom_{k\Sigma_d}(S^{(d-r, 1^r)}, S^{(d-r, 1^r)})=(r+1)/2.$

      \item If $d$ is odd then $S^{(d-r, 1^r)}$ is indecomposable if and only if $d-r-1 \equiv 0 \operatorname{mod} 2^L$ where $2^{L-1} \leq r < 2^L.$
\end{enumerate}
\end{thm}

From Theorem \ref{thm: Murphyomnibus} we can extract the following twisting result:

\begin{prop}
\label{prop: twistingp=2}
Let $p=2$ and $d \geq 2r$. Then:

\begin{enumerate}
\item $\dim \Hom_{k\Sigma_d}(S^{(d-r, 1^r)}, S^{(d-r, 1^r)})= \dim \Hom_{k\Sigma_d}(S^{(d-r+2, 1^r)}, S^{(d-r+2, 1^r)}).$

\item Suppose $2^{L-1} \leq r < 2^L$. Then $S^{(d-r, 1^r)}$ is indecomposable if and only if $S^{(d-r+2^L, 1^r)}$ is.

\end{enumerate}
\end{prop}

Since a module is indecomposable if and only if its endomorphism algebra is local, we conclude from Proposition \ref{prop: twistingp=2}(2) that the vector space isomorphisms in Proposition \ref{prop: twistingp=2}(1) are not, in general, algebra isomorphisms.

In \cite{DodgeFayers}, Dodge and Fayers discovered new infinite series of decomposable Specht modules in characteristic two, the first new examples since Murphy's 1980 paper. Again in their series we see the twisting $\lambda \rightarrow \lambda +p^r\tau$ occurring. For example a special case of Theorem 3.1 in \cite{DodgeFayers} is:

\begin{prop}
Let $p=2$. Then $S^{(4+4n,3,1,1)}$ is decomposable for $n \geq 0$.
\end{prop}

We can ask much more generally:
\begin{prob}
\label{prob: homaddpa}
Find general theorems relating $\Hom_{k\Sigma_d}(S^\lambda, S^\mu)$ with the space $\Hom_{k\Sigma_d}(S^{\lambda+(p^a)}, S^{\mu+(p^a)})$. More generally, inspired by Lyle's work, we could ask for homomorphism results relating $\lambda + p^r\tau$ and $\mu + p^r\tau$ for more general $\tau.$
\end{prob}

\noindent The work in \cite{HenkeKoenigRelatingPolynomial} mentioned in Section \ref{sec: Morita} may be relevant here for Problem \ref{prob: homaddpa}

As for comparing $\Hom_{k\Sigma_d}(S^\lambda, S^\mu)$ with $\Hom_{k\Sigma_d}(S^{p\lambda}, S^{p\mu})$ in hopes of a result like Theorem \ref{thm: Gill homYoungmodule}, the following example suggests some caution. The computations were done in GAP4 \cite{GAP4} using code written by Matthew Fayers.

\begin{exm}
\label{exm: homstabilitySpechtfails}Let $p=3$. Then:

\begin{eqnarray*}
  \dim \Hom_{k\Sigma_9}(S^{(7,1,1)}, S^{(3,1^6)}) &=  & 0 \\
  \dim \Hom_{k\Sigma_{27}}(S^{(21,3,3)}, S^{(9,3^6)}) &=  & 1 \\
\dim \Hom_{k\Sigma_{81}}(S^{(63,9,9)}, S^{(27,9^6)}) &=  & 0. \end{eqnarray*}
\end{exm}

\subsection{Generic cohomology for Specht modules} We recently proved a generic cohomology type result for Specht modules. The proof proceeds by translating the problem to $GL_n(k)$ using the result of Kleshchev and Nakano \cite[6.3(b)]{KNcomparingcohom}that:

\begin{equation}
\label{eq: Spechtcoho}
\HH^i(\Sigma_d, S^\lambda) \cong \Ext^i_{GL_d(k)}(\HH^0(d), \HH^0(\lambda), \,\, 0 \leq i \leq 2p-4.
\end{equation}
Using extensive knowledge on the structure of $\HH^0(d)$ worked out by Doty \cite{DotysubmodulestructureWeylmodules} together with knowledge of cohomology for the Borel subgroup $B$ and its Frobenius kernel $B_r$, we applied the Lyndon-Hochschild-Serre spectral sequence to obtain:

\begin{thm}
\label{thm: gencohoSpecht}
Let $p \geq 3$ and $\lambda \vdash d$. Then
$$\HH^1(\Sigma_{pd}, S^{p\lambda}) \cong \HH^1(\Sigma_{p^2d}, S^{p^2\lambda}).$$
\end{thm}
Theorem \ref{thm: gencohoSpecht} can be interpreted as a generic cohomology theorem for Specht modules in degree one. We remark that the same result holds for $\HH^0(\Sigma_d, S^\lambda)$, although somewhat trivially as $\Hom_{k\Sigma_{pd}}(k, S^{p\lambda})$ is always zero unless $\lambda=(d).$

It is natural to ask if this extends to a more general theorem like Theorem \ref{thm: Youngmodulegenericcohomology}, that is:

\begin{prob}
\label{prob: gencohoSpec}
\label{prob: genericcohoinalldegrees}
Fix $i>0$. Is there a constant $c(i)$ such that for any $d$ with $\lambda \vdash d$ and any $a \geq c(i)$ that
$$\HH^i(\Sigma_{p^ad}, S^{p^a\lambda}) \cong \HH^i(\Sigma_{p^{a+1}d}, S^{p^{a+1}\lambda})?$$
\end{prob}

As in the case of the Young module cohomology, the proof of Theorem \ref{thm: gencohoSpecht} leaves one unable to produce an explicit map, so we can ask:

\begin{prob}
Given an element $0 \rightarrow S^{p\lambda} \rightarrow M \rightarrow k \rightarrow 0$ in $\HH^1(\Sigma_{pd}, S^{p\lambda})$, can one explicitly construct an extension of $S^{p^2\lambda}$ by $k$ realizing the isomorphism in Theorem \ref{thm: gencohoSpecht}?
\end{prob}

We also proved a generic cohomology result of the other variety. Namely:
\begin{thm}
\label{thm: stabilityaddingpr}
Let $\lambda \vdash d$ and $p^r>d$. Then:

$$\HH^1(\Sigma_d, S^\lambda) = \HH^1(\Sigma_{d+p^r}, S^{\lambda+(p^r)}).$$

\end{thm}

Finally we ask for stronger results like that of Theorem \ref{thm: stabilityaddingpr}.

\begin{prob}
\label{prob: resultsaddingprmu}
Let $\lambda \vdash d$ and $\mu \vdash c$. Can one find more results that relate the cohomology $\HH^i(\Sigma_d, S^\lambda)$ and $\HH^i(\Sigma_{d+cp^r}, S^{\lambda + p^r\mu})$?
\end{prob}

We end this section by mentioning a different sort of stability result that holds for $\HH^0(\Sigma_d, S^\lambda)$ and, in all examples we have computed, also for $\HH^1(\Sigma_d, S^\lambda)$. For an integer $t$ let $l_p(t)$ be the least nonnegative integer
satisfying $t<p\,^{l_p(t)}$.
The following is an easy consequence of Theorem 24.4 in \cite{Jamesbook}:

\begin{lem}
Suppose $\lambda=(\lambda_1, \lambda_2, \ldots, \lambda_s)\vdash d$ and suppose $a \equiv -1 \operatorname{mod }  p^{l_p(\lambda_1)}$. Then
\begin{equation}
\label{eq:Homstabilizesaddingcong-1asnewfirstrow}
\HH^0(\Sigma_d, S^\lambda) \cong \HH^0(\Sigma_{d+a}, S^{(a,\lambda_1, \lambda_2, \ldots, \lambda_s)}).
\end{equation}
\end{lem}
This leads to the following

\begin{prob}
\label{prob:addingfirstrow-1stabilityonext}
Does the isomorphism in \eqref{eq:Homstabilizesaddingcong-1asnewfirstrow} hold for $\HH^i$ for any other $i>0$?
\end{prob}

\section{Complexity of symmetric group modules.}
\label{sec: Specht module complexity}

Some computer calculations done by the VIGRE algebra group at the University of Georgia suggest that twisting of partitions may arise in determining the complexity of Specht modules.  A thorough discussion of the complexity of modules can be found in \cite[Ch. 5]{Bensonrepteoryvolume2} Recall that an indecomposable module $M$ has \emph{complexity} the smallest $c=c(M)$ such that the dimensions in a minimal projective resolution are bounded by a polynomial of degree $c-1$. The maximum possible complexity for $M$ is the $p$-rank of the defect group of its block, which for the symmetric group is just the $p$-weight $w$ of the block.

Determining the complexity of various $k\Sigma_d$ modules is an active area of research. The complexity $c(Y^\lambda)$ was determined in \cite[3.3.2]{HNsupportvariety}. It is worth remarking that for $\lambda \vdash d$ it follows immediately from that result that $c(Y^{p\lambda})=d$, the maximum possible.

Very little is known on the complexity of simple modules $D^\lambda$. The paper \cite{HNsupportvariety} gives an answer when $D^\lambda$ is \emph{completely splittable.} The preprint \cite{LimTancomplexityRouquierpreprint} show that simple modules in Rouquier blocks of weight $w<p$ all have complexity $w$.

In contrast to the ``twisted" Young modules $Y^{p\lambda}$ having maximal possible complexity, it seems the situation for Specht modules $S^\lambda$ may be somewhat reversed. Say a partition $\lambda$ is $p \times p$ if both $\lambda$ and $\lambda '$ are of the form $p\tau$. Equivalently if the Young diagram of $\lambda$ is made up of $p \times p$ blocks.

 The UGA VIGRE Algebra Group made the following conjecture:

 \begin{conj}[UGA VIGRE \footnote{This conjecture and some discussion can be found at \href{http://www.math.uga.edu/~nakano/vigre/vigre.html}
{http://www.math.uga.edu/$\sim$ nakano/vigre/vigre.html}}]
\label{conj: Vigre}
Let $S^\lambda$ be in a block $B$ of weight $w$. Then the complexity of $S^\lambda$ is $w$ if and only if $\lambda$ is not $p \times p$.
\end{conj}

In \cite{HemmercomplexitySpecht} we proved that when $\lambda$ is $p \times p$ then its complexity is not maximal, by finding a natural equivalent condition for $p \times p$ in terms of the abacus display for $\lambda$, and then looking at the branching behavior of $S^\lambda$. The other (surely more difficult!) direction of the conjecture remains open:

\begin{prob}
\label{prob: Vigreconjotehrway}
Resolve the other direction of Conjecture \ref{conj: Vigre}.
\end{prob}

\begin{prob}
\label{prob: doescompmlexitydropmore}
Suppose $\lambda$ is \pbp of weight $w$. Is the complexity of $S^\lambda$ equal to $w-1$, or can it be less than $w-1$?
\end{prob}

Problem \ref{prob: doescompmlexitydropmore} has been resolved only in the case $\lambda=(p, p, \ldots, p) \vdash p^2$. In this case the support variety was computed explicitly by Lim \cite{LimvarietyforSomeSpechtmodules}, and its dimension (which equals the complexity) is indeed $p-1$. The support variety for the Specht module $S^{(3,3,3)}$ in characteristic three provides a motivating example in Chapter 7 of the book \cite{BensonElementaryAbelianbookpreprint} as a small-dimensional module with a very interesting support variety.
Perhaps further twisting might lower the complexity even more:

\begin{prob}One can generalize the definition of \pbp in several ways. For example one obvious generalization would be to require $\lambda$ be $p^2 \times p^2$ . Can one say anything interesting about these situations? Perhaps the complexity drops by even more in this case? For example can one determine the complexity of $S^{(9^9)}$ in characteristic three?
\end{prob}

\section{Generic cohomology for simple modules and twists of the Mullineux map}
Recall that the irreducible modules for $k\Sigma_d$ are labeled by $p$-regular partitions and are denoted $D^\lambda$. However there is another indexing, by $p$-restricted partitions and denoted $D_\mu$. The latter is more natural in some ways, as $D_\mu$ is the image under the Schur functor of the irreducible $G$ module $L(\mu)$. The labellings  are related by:
\begin{equation}
\label{eq: simplemodulelabelsrelations}
D^\lambda \otimes \sgn \cong D_{\lambda '}.
\end{equation}

It was a longstanding open problem to determine the partition $m(\lambda)$ so that $D^\lambda \otimes \sgn \cong D^{m(\lambda)}$. This problem was finally solved by Kleshchev in \cite{KleshchevBranchingIII}. A short time later Ford and Kleshchev \cite{FordKleshMullineux} confirmed that Kleshchev's answer agreed with the original conjecture made by Mullineux in \cite{MullineuxBijectionsofpregular}. We will find it useful to use Mullineux's original algorithm and also a different (but of course equivalent) description given later by Xu in \cite{XuMullineuxCIApaper}. Both are nicely described in \cite{BessenrodtOlssonXu}. It follow from \eqref{eq: simplemodulelabelsrelations} that:

\begin{equation}
\label{eq: relaateupdownMull}
D^\lambda \cong D_{m(\lambda)'}
\end{equation}
so one can easily arrive at a version of the Mullineux bijection, except on $p$-restricted partitions; namely $\lambda \rightarrow m(\lambda') '.$

\subsection{Generic cohomology for two-part irreducibles}
Suppose one wanted a generic cohomology theorem for extensions between simple $k\Sigma_d$ modules. The simple module $L(\lambda)$ corresponds to $D_\lambda$, but $p\lambda$ is never $p$-restricted so $D_{p\lambda}$ does not exist. Strangely though something seems to be going on with the ``wrong" upper notation. For example:

\begin{prop}
\label{prop: stability for 2part simples} Assume $p>2$. Let $\lambda=(v,u)$ and $\mu=(s,r)$ be partitions of $d$. Then
$$\Ext^1_{k\Sigma_{pd}}(D^{p\lambda}, D^{p\mu}) \cong \Ext^1_{k\Sigma_{p^2d}}(D^{p^2\lambda}, D^{p^2\mu}).$$
\end{prop}

\begin{proof}Assume $u \geq r$ without loss of generality. All the extensions between two-part simple modules were worked out by Kleshchev and Sheth in \cite{KlesSheth}  (but see the Corrigendum \cite{KleshchevShethCorrgiendum}). In their notation we have $pv-pu+1 = 1 + \sum_{i \geq 1}a_ip^i$ and the condition for the $\Ext$ group to be nonzero is that $pu-pr=(p-a_i)p^i$ for some $i$ such that $a_i >0$ and either $a_{i+1} < p-1$ or $u < p^{i+1}$. This condition is clearly equivalent to the corresponding one for $p^2v-p^2u+1$ and $p^2u-p^2r.$
\end{proof}

 We remark that Proposition \ref{prop: stability for 2part simples} requires the additional twist before the cohomology stabilizes. For example when $p=3$ one can use Kleshchev-Sheth's result to compute:

\begin{eqnarray*}
  \Ext^1_{\Sigma_{29}}(D^{(20,9)}, D^{(26,3)}) &\cong & k\\
  \Ext^1_{\Sigma_{87}}(D^{(60,27)}, D^{(78,9)}) &\cong& 0.
\end{eqnarray*}

Of course this suggests the following:

\begin{conj}
\label{conj: gencohosimples}Let $\lambda, \mu \vdash d$. Then:
$$\Ext^1_{k\Sigma_{pd}}(D^{p\lambda}, D^{p\mu}) \cong \Ext^1_{k\Sigma_{p^2d}}(D^{p^2\lambda}, D^{p^2\mu}).$$
\end{conj}

Once again we have a vector space isomorphism in cohomology (Proposition \ref{prop: stability for 2part simples}) with no module homomorphisms realizing it!

\subsection{Mullineux map and twists}
The labelling of irreducibles $D_\mu$ by $p$-restricted partitions seems more natural when comparing with $GL_n(k)$ (where actual Frobenius twists can occur). But $p\lambda$ is never $p$-restricted, and we observed above a relationship between $D^{p\lambda}$ and $D^{p^2\lambda}$. Using \eqref{eq: relaateupdownMull} suggests some relationship between $m(p\lambda)'$ and $m(p^2\lambda)'$. This led us to a strictly  combinatorial question, namely is there any relation between twisting and the Mullineux map? And then of course given any such relation, is there a representation-theoretic interpretation? We have found a large class of partitions that have interesting behavior here.

For example if $\lambda=(\lambda_1, \lambda_2, \ldots, \lambda_s) \vdash d$ is a partition with distinct parts, define

$$\hat{\lambda}=(\lambda_1^{p-1}, \lambda_2^{p-1}, \ldots, \lambda_s^{p-1}) \vdash (p-1)d.$$

\begin{prop}
\label{prop: mulllambdahat}
Suppose $\lambda=(\lambda_1, \lambda_2, \ldots, \lambda_s) \vdash d$ has distinct parts. Then $m(\hat{\lambda})=(p-1)\lambda$.
\end{prop}

\begin{proof}
Consider the $p$-regular version of Xu's algorithm from \cite{XuMullineuxCIApaper} (nicely described in \cite{BessenrodtOlssonXu}). If we apply it to calculate ${\hat{\lambda}}^X$ using \cite[Def 3.5]{BessenrodtOlssonXu} we obtain $j_1=j_2= \cdots = j_{p-1}=s$. At this point in the algorithm the first column (consisting of $s(p-1)$ nodes) will have been removed from $\hat{\lambda}$, and what remains is $\widehat{\overline{\lambda}}$ where $\overline{\lambda}$ denotes $\lambda$ with its first column removed. One can now apply induction using \cite[Proposition 3.6($2'$)]{BessenrodtOlssonXu} or just continue with the algorithm to obtain $(p-1)\lambda.$
\end{proof}

\begin{cor}
\label{cor: m(twist)=twistm}
Let $\lambda \vdash d$ have distinct parts. Then
$$m(p\mu)=pm(\mu)$$ for both $\mu=(p-1)\lambda$ and $\mu=\hat{\lambda}.$
\end{cor}
\begin{proof}
By Proposition \ref{prop: mulllambdahat} we have:
\begin{eqnarray*}
  m((p)(p-1)\lambda) &=& \widehat{p\lambda} \\
   &=& p\hat{\lambda} \\
   &=&pm((p-1)\lambda)
\end{eqnarray*}
and
\begin{eqnarray*}
  m(p\hat{\lambda}) &=& m(\widehat{p\lambda}) \\
   &=& p(p-1)\lambda\\&=&pm(\hat{\lambda}).
\end{eqnarray*}

\end{proof}
The appearance above of $p(p-1)\lambda$ for $\lambda$ having distinct parts is reminiscent of the twist of the Steinberg weight $p(p-1)\rho$ from the $GL_n(k)$ theory, although we have no representation-theoretic interpretation at this time.

The examples in Corollary \ref{cor: m(twist)=twistm} are not the only ones where $m(p\mu)=pm(\mu)$ although other examples seem to be rare. For example if $p=5$ and $d=20$ then Corollary \ref{cor: m(twist)=twistm} yields all $\mu \vdash 20$ where $m(5\mu)=5m(\mu)$. On the other hand for $p=5$ and $\lambda=(3,3)$ we have $m(15,15)=(10,10,10)=5m(3,3)$, although $\lambda$ does not have distinct parts.. This brings us to:

\begin{prob}
\label{prob: mulloftwists} Classify all $\lambda$ such that $m(p\lambda)=pm(\lambda)$ and give a representation-theoretic interpretation of the answer.
\end{prob}
The subset of such $\lambda$ arising in Corollary \ref{cor: m(twist)=twistm} is certainly closed under twisting, as are all other examples we have computed, so we conjecture:

\begin{conj}
Suppose $m(p\lambda)=pm(\lambda)$. Then $m(p^2\lambda)=pm(p\lambda)$.
\end{conj}

More generally we can ask:

\begin{prob}
Classify all $\lambda$ such that $m(p\lambda)=p\tau$ for some $\tau.$
\end{prob}
For example when $p=5$ and $\lambda=(6,6,4)$ we have:
$$m(30,30,20)=(20^3,5^4)=5m(10,3,3).$$

Even for partitions without such a nice relationship, there still seems to be some intriguing behavior among the $m(p^a\lambda)$ for various $a$. For example we will now derive a sort of Steinberg tensor product theorem for Mullineux maps.

First define $\tau_n=m(1^n)$, so the trivial $k\Sigma_n$ module is $D_{\tau_n}$. It is well known that $\tau_n=(p-1,p-1, p-1, \cdots, p-1, a)$.

\begin{lem}
\label{lem: distinct parts}
Suppose $\lambda=(\lambda_1, \lambda_2, \ldots, \lambda_s)$ is a partition with distinct parts. Then:
$$m(p\lambda)'= \tau_{p\lambda_1} + \tau_{p\lambda_2} + \cdots + \tau_{p\lambda_s}.$$
\end{lem}

For example if $\lambda=(4,2,1)$ and $p=5$ then $$m(5\lambda)' = (12, 9, 6, 4, 4)=(4^5)+(4,4,2)+(4,1)=\tau_{20} + \tau_{10} + \tau_5.$$

\begin{proof}
The proof is by induction on $s$ where the case $s=1$ is just the definition of $\tau_n$. We will use the original algorithm of Mullineux,  described and proved in \cite[p. 272]{FordKleshMullineux}. Since the parts of $\lambda$ are distinct, then all the rim $p$-hooks removed in determining the Mullineux symbol $G_p(p\lambda)$ are horizontal. Thus we determine the Mullineux symbol $G_p(m(p\lambda))= $

\begin{equation}
\label{eq: Mullsymbol}
\left(
                    \begin{array}{llllllllll}
                      sp  & \cdots & sp & (s-1)p  & \cdots & (s-1)p &  \cdots & p  & \cdots & p \\
                      s(p-1) & \cdots & s(p-1)  & (s-1)(p-1) & \cdots & (s-1)(p-1) & \cdots & p-1 & \cdots & p-1 \\
                    \end{array}
                  \right)
\end{equation}

\noindent where the first column in \eqref{eq: Mullsymbol} occurs $\lambda_s$ times, the second occurs $\lambda_{s-1}-\lambda_s$ times, the third $\lambda_{s-2}-\lambda_{s-1}$, etc. So the last column
 occurs $\lambda_1-\lambda_2$ times. Notice then that removing the columns $\begin{array}{l}
sp \\
 s(p-1)
 \end{array}$
 we obtain the Mullineux symbol for $p(\lambda_1-\lambda_s, \lambda_2-\lambda_s, \ldots, \lambda_{s-1}-\lambda_s)$. So we apply induction and the result follows.

\end{proof}

Using the result above and drawing a diagram with the appropriate $\tau_{p\lambda_i}$ and $\tau_{p^2\lambda_i}$, the following corollary is immediate:
\begin{cor}
\label{cor: steinbergmull}
Suppose $\lambda$ has distinct parts. Then:
$$m(p^2\lambda)-m(p\lambda)=p\hat{\lambda}.$$
\end{cor}
Using Proposition \ref{prop: mulllambdahat} one can rewrite this as:
$$m(p^2\lambda)-m(p\lambda)=pm((p-1)\lambda).$$

For arbitrary $\lambda$ with repeated parts we conjecture a weaker form of stability after multiple ``twists:"
\begin{conj}
\label{conj: Mullmultipletwists}
Let $\lambda \vdash d$. Then there exist $1 \leq a<b$ such that:
$$m(p^b\lambda)=m(p^a\lambda) + p^a\tau.$$
\end{conj}
As an example of Conjecture \ref{conj: Mullmultipletwists} needing multiple twists consider the following:

\begin{exm}
\label{exm: mullmanytwists}
Let $p=7$ and $\lambda=(29^2, 24, 4^2, 3^3, 2, 1).$ Then:

$$m(7^5\lambda)-m(7\lambda)=7(123840^5, 9600^5, 5400^4, 3840^5, 800^6, 400^6).$$ If $1\leq x<y <5$ then $m(7^y\lambda)-m(7^x\lambda)$ is not of the form $7\tau$, i.e. this stability really requires at least five twists.

\end{exm}

Finally we remark that the results above on the Mullineux map did not depend on $p$ being prime, and any possible representation-theoretic interpretations might hold true for the Hecke algebra of type $A$ at an $e$th root of unity.

It seems fitting to close with a completely general (and quite possibly absurd) question:

\begin{prob}
Can one say anything interesting about how the structure of the principal block $B_0(k\Sigma_{pd})$ is reflected inside $B_0(k\Sigma_{p^2d})$?
\end{prob}
For example Theorem \ref{thm: Gill homYoungmodule} is a statement about Cartan invariants.

\bibliography{references0911}

\bibliographystyle{amsplain}
\end{document}